\documentclass{amsart}
\usepackage{enumerate,amssymb,amsmath,graphics,epsfig}
\usepackage{xspace}
\usepackage{mathrsfs}

\newcommand\RE{\mathbb{R}}
\newcommand\NA{\mathbb {N}}
\renewcommand\div{\mathop{\rm{div}}\nolimits}
\newcommand\Span{\mathrm{span}}
\newcommand\Grad{\mathop{\boldsymbol\nabla}\nolimits}
\newcommand\Grads[1]{\varepsilon(#1)}
\newcommand\Huo{\HH^1_0(\Omega)}
\newcommand\Hub{\HH^1(\B)}
\newcommand\Ld{\L^2(\Omega)}
\newcommand\Ldo{L^2_0(\Omega)}
\newcommand\Ldb{\mathbf{L}^2(\B)}
\newcommand\Kt{\K_t}
\newcommand\Of{\Omega^f}
\newcommand\Os{\Omega^s}
\newcommand\OO{{0,\Omega}}
\newcommand\OB{{0,\B}}

\newcommand\Xb{\overline\X}
\newcommand\Jb{\overline J}
\newcommand\lf{\lambda_f}
\newcommand\ls{\lambda_s}

\renewcommand\c{\mathbf{c}}

\newcommand\n{\mathbf{n}}
\newcommand\s{\mathbf{s}}
\renewcommand\u{\mathbf{u}}

\renewcommand\v{\mathbf{v}}
\newcommand\w{\mathbf{w}}
\newcommand\x{\mathbf{x}}
\newcommand\y{\mathbf{y}}
\newcommand\z{\mathbf{z}}
\newcommand\X{\mathbf{X}}
\renewcommand\H{\mathbf{H}_0}
\newcommand\HH{\mathbf{H}}
\renewcommand\L{\mathbf{L}}
\newcommand\Y{\mathbf{Y}}
\newcommand\V{\mathbf{V}_0}
\newcommand\W{\mathbf{W}}
\newcommand\B{\mathcal B}
\newcommand\ssigma{\boldsymbol{\sigma}}
\newcommand\llambda{\boldsymbol{\lambda}}
\newcommand\mmu{{\boldsymbol{\mu}}}
\newcommand\ppsi{\boldsymbol{\psi}}
\newcommand\pphi{\boldsymbol{\varphi}}
\newcommand\cchi{\boldsymbol{\chi}}
\newcommand\aalpha{\boldsymbol{\alpha}}
\newcommand\bbeta{\boldsymbol{\beta}}
\newcommand{\F}{\mathbb{F}}
\newcommand{\I}{\mathbb{I}}
\newcommand\K{\mathbb{K}}
\renewcommand{\P}{\mathbb{P}}
\newcommand\ucX{\u(\X(\cdot,t),t)}
\newcommand\vcX{\v(\X(\cdot,t))}
\newcommand\ucXb{(\u\circ\Xb)(t)}
\newcommand\vcXb{\v\circ\Xb}
\newcommand\dr{\delta_\rho}
\newcommand\alm{\alpha^{(m)}}
\newcommand\bem{\beta^{(m)}}
\newcommand\aalm{\aalpha^{(m)}}
\newcommand\bbem{\bbeta^{(m)}}
\newcommand\weakstar{\stackrel{\ast}{\rightharpoonup}}

\theoremstyle{plain}
\newtheorem{thm}{Theorem}
\newtheorem{proposition}[thm]{Proposition}
\newtheorem{lemma}[thm]{Lemma}

\newtheorem{problem}{Problem}
\newtheorem{ass}{Assumption}
\theoremstyle{remark}

\begin{document}
\title[Existence and uniqueness for an FSI problem]
{On the existence and the uniqueness of the solution to a
fluid-structure interaction problem}
\author{Daniele Boffi}
\address{King Abdullah University of Science and Technology (KAUST), Saudi
Arabia and University of Pavia, Italy}
\email{daniele.boffi@kaust.edu.sa}
\urladdr{https://cemse.kaust.edu.sa/people/person/daniele-boffi}
\author{Lucia Gastaldi}
\address{DICATAM, Universit\`a di Brescia, Italy}
\email{lucia.gastaldi@unibs.it}
\urladdr{http://lucia-gastaldi.unibs.it}

\begin{abstract}
In this paper we consider the linearized version of a system of partial
differential equations arising from a fluid-structure interaction model.
We prove the existence and the uniqueness of the solution under natural
regularity assumptions.
\end{abstract}
\maketitle
\section{Introduction}
\label{se:intro}

The mathematical analysis and the numerical approximation of problems
involving the interaction of fluids and solids are essential for the modeling
and simulation of a variety of applications related to engineering, physics,
and biology.

We consider a model presented in~\cite{BCG,BG} based on a fictitious domain
approach and the use of a distributed Lagrange multiplier. The considered
formulation is the evolution of a model originated from a finite element
approach of the immersed boundary method~\cite{BG_Bathe,BGHP,BCG_rho}. The
immersed boundary method has been introduced by Peskin and his collaborators
in several seminal papers~\cite{Peskin77,PeMcQ89,MQPe2000,PeAN} where the
interaction between the fluid and the solid was modeled by a suitably defined
Dirac delta function and the numerical approximation was performed by finite
differences.
One of the differences of the immersed boundary method with respect to other
possible approaches is that the discretization is performed by using two fixed
meshes: one for the fluid domain (artificially extended to include the
immersed solid) and one for the reference configuration of the solid. This
choice has the advantage that the computational meshes need not be updated at
each time step.  On other hand, the intersection between the fluid mesh and
the image of the solid mesh into the actual solid configuration needs to be
evaluated; in our approach the Lagrange multiplier is responsible for such
coupling.
The numerical analysis of the problem shows appealing properties related to
its stability~\cite{newren,BCG} and the numerical investigations demonstrate
the superiority of the finite element approach with respect to the original
finite difference scheme in terms of mass conservation.
Higher order time discretization has been investigated in~\cite{wolf}.

In this paper we address the study of the existence and uniqueness of the
continuous solution. The solution has four components: fluid velocity $\u$ and
pressure $p$ (extended into the solid region in the spirit of the fictitious
domain), the position of the solid domain inside the fluid, seen as a mapping
$\X$ from a reference configuration, and the Lagrange multiplier $\llambda$
supported in the solid reference domain that is used to enforce the coupling
between the solid and the fluid.
The problem is highly non linear; in particular the unknown $\X$ defines
mathematically the region occupied by the solid at a given time.
We consider a linearization of the problem with respect to the variable $\X$
and, for simplicity, we neglect the convective term of the Navier--Stokes
equations.

Our existence and uniqueness proof is based on a Faedo--Galerkin
approximation as done in~\cite{T} for the study of the Navier--Stokes
equations. We extend the results of~\cite{DGHL} where the coupling of the
incompressible Navier--Stokes equations with a linear elasticity model in a
fixed domain is considered.

Existence and uniqueness results for models related to fluid-structure
interactions have a limited but not empty occurrence in the literature.
In particular, some authors discussed the existence of weak solutions in the
case of a fluid containing rigid solids or elastic bodies whose behavior is
described by a finite number of 
modes~\cite{CSMT,DE1999,DE2000,DEGLT,F,GM,GLS,HS,Serre,TaTu,Ta}.
Other results are available for the existence of weak solutions in the case of
a fluid enclosed in a solid
membrane~\cite{Beirao,CDEG,MC2013a,MC2013b,MC2014,MC2015,MC2016} or interacting
with a plate~\cite{FO}; 
the typical example of application is the blood flowing in a vessel~\cite{QTV}.
The existence of the solution in the case of viscoelastic particles immersed
in a Newtonian fluid is discussed in~\cite{LW} using the Eulerian description
for both fluid and solid.
Local-in-time existence and uniqueness of strong solutions for a model
involving an elastic structure immersed in a fluid is analyzed
in~\cite{Coutand2005,Coutand2006,RaymondVanni2014,Boulakia2017,Boulakia2019}.

In the next section we recall the strong formulation of our model.
Section~\ref{se:FD_DLM} presents the fictitious domain approach together with
its variational formulation. The linearized problem is described in
Section~\ref{se:linearized} and the main existence and uniqueness result for
the velocity $\u$ and the position of the solid $\X$ is stated and proved in
Section~\ref{se:main}. Finally, Section~\ref{se:pl} is devoted to the
existence and uniqueness of the pressure $p$ and the multiplier $\llambda$.

\section{Setting of the problem}

In this section we recall the formulation of the fluid-structure interaction
problem presented in~\cite{BG}. We assume that we are given a two or three
dimensional Lipschitz and convex domain $\Omega$ which is occupied by a fluid
and a solid.  We denote by $\Of_t$ and $\Os_t$ the regions where the fluid and
the solid are respectively located at time $t$, so that $\Omega$ is the
interior of $\overline\Omega^f_t\cup\overline\Omega^s_t$. The regularity of
the two subdomains will be made more precise later on as a consequence of
Assumption~\ref{ass}. For simplicity we assume that
$\partial\Os_t\cup\partial\Omega$ is empty, that is the solid is immersed in
the fluid, and the moving interface $\partial\Of_t\cap\partial\Os_t$ is
denoted by $\Gamma_t$.

We denote by $\u$, $\ssigma$, and $\rho$ the velocity, stress tensor, and mass
density, respectively, and we use subscripts $f$ or $s$ to refer to the fluid
or to the solid. We assume that the densities $\rho_f$ and $\rho_s$ are
positive constants.

The following equations represent the strong form of the problem we are
interested in, corresponding to the interaction of an incompressible fluid and
an incompressible immersed elastic structure.

\begin{equation}
\label{eq:general}
\aligned
&\rho_f\dot\u_f=\div\ssigma_f  &&\text{in }\Of_t\\
&\div\u_f=0 &&\text{in }\Of_t\\
&\rho_s\dot\u_s=\div\ssigma_s &&\text{in }\Os_t\\
&\div\u_s=0 &&\text{in }\Os_t\\
&\u_f=\u_s &&\text{on }\Gamma_t\\
& \ssigma_f\n_f=-\ssigma_s\n_s &&\text{on }\Gamma_t.
\endaligned
\end{equation}
The following initial and boundary conditions are imposed on $\partial\Omega$.
\begin{equation}
\label{eq:ibc}
\aligned
&\u_f(0)=\u_{f0} && \text{on }\Of_0,\\
&\u_s(0)=\u_{s0} && \text{on }\Os_0,\\
&\u_f(t)=0 && \text{on }\partial\Omega.
\endaligned
\end{equation}
The fluid stress tensor is defined by the Navier--Stokes law as it is common
for Newtonian fluids
\begin{equation}
\label{eq:sigmaf}
\ssigma_f=-p_f\mathbb{I}+\nu_f\Grads{\u_f},
\end{equation}
where $\Grads\u=(1/2)\left(\Grad\u+(\Grad\u)^\top\right)$ is the symmetric
gradient and $\nu_f$ represents the viscosity of the fluid.

The solid domain $\Os_t$ is the image of a reference domain $\B=\Os_0$.
Let $\X(t):\B\to\Os_t$ be the mapping that associates to each point $\s\in\B$ a
point $\x\in\Os_t$. When it is needed in order to avoid confusion we will use
the notation $\X(\s,t)$ to denote the dependence on both space and time.
We assume that $\X$ is one to one and that, for all $t\in[0,T]$,
$\|\X(\s_1,t)-\X(\s_2,t)\|\ge\gamma\|\s_1-\s_2\|$ for all $\s_1,\s_2\in\B$ for
a positive constant $\gamma$.
In particular, $\X(t)$ is invertible with Lipschitz inverse. 

We denote by $\F=\Grad_s\X$ the deformation gradient; its determinant is
denoted by $|\F|$. We have that $|\F|$ is constant in time since the fluid and
the solid are incompressible; it is not restrictive to assume that
$\X_0(\s)=\s$ for $\s\in\B$, so that $|\F|=1$ for all $t$.

When dealing with moving domains it is essential to be precise with respect to
the Eulerian and Lagrangian descriptions of the involved quantities.
In~\eqref{eq:general}, we used the dot over the velocity in order to denote
the material time derivative. The Eulerian description of the fluid gives
$\dot\u_f=\partial\u_f/\partial t+\u_f\cdot\Grad\u_f$.
The spatial description of the material velocity in the solid, where the
Lagrangian representation is used, reads
\begin{equation}
\label{eq:materialvelocity}
\u_s(\x,t)=\frac{\partial\X(\s,t)}{\partial t}\Big|_{\x=\X(\s,t)}
\end{equation}
so that $\dot\u_s(\x,t)=\partial^2\X(\s,t)/\partial t^2|_{\x=\X(\s,t)}$.

Following what we did in~\cite{BGHP}, we consider a viscous-hyperelastic solid
structure and we define the Cauchy stress tensor as the sum of two
contributions $\ssigma_s=\ssigma_s^f+\ssigma_s^s$.
There is a fluid-like part of the stress
\begin{equation}
\label{eq:sigmas}
\ssigma_s^f=-p_s\mathbb{I}+\nu_s\Grads\u_s
\end{equation}
for a positive constant viscosity $\nu_s$ and an elastic part $\ssigma_s^s$.
The elastic part of the stress $\ssigma_s^s$ can be written in terms of the
first Piola--Kirchhoff stress tensor $\P$ with a change of variables from the
Eulerian to the Lagrangian framework:
\begin{equation}
\label{eq:Piola}
\P(\F(\s,t))=
|\F(\s,t)|\ssigma_s^s(\x,t)\F^{-\top}(\s,t)\qquad\text{for }\x=\X(\s,t).
\end{equation}
Following~\cite{BG} we are going to consider a linear dependence of $\P$ on
$\Grad_s\F$, namely
\begin{equation}
\label{eq:linearelasticity}
\P(\F)=\kappa\Grad_s\F.
\end{equation}

The following notation will be used throughout the paper.

If $D$ is a domain in $\RE^d$, we denote by
$W^{s,p}(D)$ the Sobolev space on $D$ ($s\in\RE$, $1\le p\le\infty$), and by
$\|\cdot\|_{s,p,D}$ its norm (see, for example,~\cite{Adams}). As usual we write
$H^s(D)=W^{s,2}(D)$ and omit $p$ in the norm and seminorm when
it is equal to $2$. Moreover, bold characters denote vector valued
functions and the corresponding functional spaces. The dual space of a Hilbert
space $X$ will be denoted with $X'$. The notation
$(\cdot,\cdot)_D$ stands for the scalar product in $L^2(D)$ and the duality
pairing is denoted by brackets $\langle\cdot,\cdot\rangle$.
The subscript indicating the domain is omitted if the domain is
$\Omega$, while we shall always use it for the reference domain
$\B$.
We will make use of the space $H^1_0(D)$ of functions in $H^1(D)$ with zero
trace on the boundary of $D$ and of its dual $H^{-1}(D)$. Moreover, $L^2_0(D)$
denotes the subspace of $L^2(D)$ of functions with zero mean value on $D$.

We denote by $\mathscr{D}(D)$ the space of $C^\infty$ functions
with compact support in $D$.
When $X$ is a Banach space, we denote by $L^p(0,T;X)$ ($1\le p\le\infty$) the
space of $L^p$-integrable functions from $(0,T)$ into $X$, which is a Banach
space with the norm
\[
\|v\|_{L^p(X)}=\left(\int_0^T\|v(t)\|_X^p\,dt\right)^{1/p}.
\]
Analogously, the space $C^m([0,T];X)$ denotes the space of functions from
$[0,T]$ to $X$ which are continuous up to the $m$-th derivative in $t$.

Finally, we are going to use the following spaces:
\begin{equation}
\label{eq:divfreespaces}
\aligned
&\mathscr{V}_0=\{\v\in\mathscr{D}(\Omega)^d: \div\v=0\}\\
&\H=\text{ the closure of }\mathscr{V}_0\text{ in }\Ldo\\
&\V=\text{ the closure of }\mathscr{V}_0\text{ in }\Huo.
\endaligned
\end{equation}
\section{Fictitious domain approach and Lagrange multiplier}
\label{se:FD_DLM}

We extend the fluid velocity and the pressure into the solid domain
by introducing new unknowns with the following meaning:
\begin{equation}
\label{eq:FDunknowns}
\u=\left\{
\begin{array}{ll}
\u_f&\text{ in } \Of_t\\
\u_s&\text{ in } \Os_t
\end{array}
\right.\qquad
p=\left\{
\begin{array}{ll}
p_f&\text{ in } \Of_t\\
p_s&\text{ in } \Os_t
\end{array}.
\right.
\end{equation}
The condition that the material velocity of the solid is equal to the
velocity of the fictitious fluid is expressed by
\begin{equation}
\label{eq:constraint}
\frac{\partial\X(\s,t)}{\partial t}=\u(\X(\s,t),t) \quad\text{for }\s\in\B.
\end{equation}
We introduce the following bilinear form:
\begin{equation}
\label{eq:defc}
\c(\mmu,\z)=\left(\Grad_s\mmu,\Grad_s\z\right)_\B+\left(\mmu,\z\right)_\B
\quad\forall\mmu,\ \z\in\Hub.
\end{equation}
It is obvious that for all $\mmu,\z\in\Hub$
\[
\aligned
&\c(\z,\z)=\|\z\|^2_{1,B}=\|\z\|^2_{0,\B}+\|\Grad_s\z\|^2_{0,B}\\
&\c(\mmu,\z)\le\|\mmu\|_{1,B}\|\z\|_{1,B}\\
&\c(\mmu,\z)=0 \text{ for all } \mmu\in\Hub \text{ implies } \z=0.
\endaligned
\]
System~\eqref{eq:general} can be formulated as follows.
\begin{problem}
\label{pb:pbvar}
Given $\u_0\in\Huo$, $\u_{s0}\in\mathbf{H}^1(\Os_0)$, and $\X_0(\s)=\s$ for
$\s\in\B$,
for almost every $t\in]0,T]$, find
$(\u(t),p(t))\in\Huo\times\Ldo$, $\X(t)\in\Hub$,
and $\llambda(t)\in\Hub$ such that it holds
\begin{subequations}
\begin{alignat}{2}
&\rho_f\frac d {dt}(\u(t),\v)+b(\u(t),\u(t),\v)+a(\u(t),\v)\qquad\notag&&\\
&\qquad-(\div\v,p(t))+\c(\llambda(t),\vcX)=0
   &&\quad\forall\v\in\Huo
     \label{eq:NS1_DLM}\\
&(\div\u(t),q)=0&&\quad\forall q\in\Ldo
     \label{eq:NS2_DLM}\\
&\dr\left(\frac{\partial^2\X}{\partial t^2}(t),\z\right)_{\B}
  +\kappa(\Grad_s\X(t),\Grad_s\z)_{\B}-\c(\llambda(t),\z)=0
  &&\quad\forall\z\in\Hub
     \label{eq:solid_DLM}\\
&\c\left(\mmu,\ucX-\frac{\partial\X}{\partial t}(t)\right)=0 
  &&\quad\forall\mmu\in \Hub
     \label{eq:ode_DLM}\\
&\u(0)=\u_0\quad\mbox{\rm in }\Omega,
     \label{eq:ciu_DLM}\\
&\X(0)=\X_0\quad\mbox{\rm in }\B,\qquad \frac{\partial\X}{\partial t}(0)=\u_{s0}
\quad\mbox{\rm in }\B.
     \label{eq:ciX_DLM}
\end{alignat}
\end{subequations}
\end{problem}
Here $\dr=\rho_s-\rho_f$ and
\[
\aligned
&a(\u,\v)=(\nu\Grads{\u},\Grads{\v})\quad \text{with }
\nu=\left\{\begin{array}{ll}
\nu_f&\text{in }\Of_t\\
\nu_s&\text{in }\Os_t
\end{array}\right.
\\
&b(\u,\v,\w)=
\frac{\rho_f}2\left((\u\cdot\Grad\v,\w)-(\u\cdot\Grad\w,\v)\right).
\endaligned
\]
We assume that $\nu\in L^\infty(\Omega)$ and that there exists a positive
constant $\nu_0>0$ such that $\nu\ge\nu_0>0$ in $\Omega$, hence the following
Korn's inequality holds true for all $\u\in\Huo$
\begin{equation}
\label{eq:Korn}
a(\u,\u)\ge \mathbf{k}\|\Grad\u\|^2_\OO.
\end{equation}

We add the following compatibility conditions for the initial velocity 
\begin{equation}
\label{eq:compatibility}
\div\u_0=0,\quad\text{and}\quad \u_0|_{\Os_0}=\u_{s0}.
\end{equation}
The second condition is related to the fact that we are assuming $\B=\Os_0$.

\section{Linearized problem}
\label{se:linearized}
We fix a function $\Xb$ which satisfies the following assumption.
\begin{ass}
\label{ass}
Let $\Xb\in C^{1}([0,T];\W^{1,\infty}(\B))$ be invertible with Lipschitz
inverse for all $t\in[0,T]$, with $\Xb(\s,0)=\s$ for $\s\in\B$.
In addition, we assume that
\begin{equation}
\label{eq:assX}
\Jb(t)=\textrm{det}(\Grad_s\Xb(t))=1 \quad\text{for all }t.
\end{equation}
\end{ass}

From now one we are going to neglect the convective term so that our problem
will read as follows.

\begin{problem}
\label{pb:pblin}
Given $\u_0\in\Huo$ and $\u_{s0}\in\Hub$,
for almost every $t\in]0,T]$ find
$(\u(t),p(t))\in\Huo\times\Ldo$, $\X(t)\in\Hub$,
and $\llambda(t)\in\Hub$ such that it holds
\begin{subequations}
\begin{alignat}{2}
&\rho_f\frac d {dt}(\u(t),\v)+a(\u(t),\v)-(\div\v,p(t))\qquad\notag&&\\
&\qquad+\c(\llambda(t),\vcXb)=0
   &&\quad\forall\v\in\Huo
     \label{eq:NS1_lin}\\
&(\div\u(t),q)=0&&\quad\forall q\in\Ldo
     \label{eq:NS2_lin}\\
&\dr\left(\frac{\partial^2\X}{\partial t^2}(t),\z\right)_{\B}
  +\kappa(\Grad_s\X(t),\Grad_s\z)_{\B}-\c(\llambda(t),\z)=0
  &&\quad\forall\z\in\Hub
     \label{eq:solid_lin}\\
&\c\left(\mmu,\ucXb-\frac{\partial\X}{\partial t}(t)\right)=0
  &&\quad\forall\mmu\in \Hub
     \label{eq:ode_lin}\\
&\u(0)=\u_0\quad\mbox{\rm in }\Omega,
     \label{eq:ciu_lin}\\
&\X(0)=\X_0\quad\mbox{\rm in }\B,\qquad \frac{\partial\X}{\partial t}(0)=\u_{s0}
\quad\mbox{\rm in }\B.
     \label{eq:ciX_lin}
\end{alignat}
\end{subequations}
\end{problem}

In the previous equations we used the notation $\vcXb=\v(\Xb(\cdot,t))$ and
$\ucXb=\u(\Xb(\cdot,t),t)$.

Let us split the second order in time Equation~\eqref{eq:solid_lin} into a
system of two differential equations of first order in time by introducing a
new unknown $\w=\frac{\partial\X}{\partial t}$. Then Problem~\ref{pb:pblin}
becomes:
\begin{problem}
\label{pb:linI}
Given $\u_0\in\Huo$ and $\u_{s0}\in \Hub$, for almost every $t\in]0,T]$ find
$(\u(t),p(t))\in\Huo\times\Ldo$, $(\X(t),\w(t))\in\Hub\times\Hub$,
and $\llambda(t)\in\Hub$ such that it holds
\begin{subequations}
\begin{alignat}{2}
&\rho_f\frac d {dt}(\u(t),\v)+a(\u(t),\v)-(\div\v,p(t))\qquad\notag&&\\
&\qquad+\c(\llambda(t),\vcXb)=0
   &&\quad\forall\v\in\Huo
     \label{eq:NS1_linI}\\
&(\div\u(t),q)=0&&\quad\forall q\in\Ldo
     \label{eq:NS2_linI}\\
&\dr\left(\frac{\partial\w}{\partial t}(t),\z\right)_{\B}
  +\kappa(\Grad_s\X(t),\Grad_s\z)_{\B}-\c(\llambda(t),\z)=0
  &&\quad\forall\z\in\Hub
     \label{eq:solid_linI}\\
&\left(\frac{\partial\X}{\partial t}(t),\y\right)_{\B}=(\w(t),\y)_{\B} 
  &&\quad\forall\y\in\Ldb
     \label{eq:vel_linI}\\
&\c\left(\mmu,\ucXb-\w(t)\right)=0 &&\quad\forall\mmu\in \Hub
     \label{eq:ode_linI}\\
&\u(0)=\u_0\quad\mbox{\rm in }\Omega,
     \label{eq:ciu_linI}\\
&\X(0)=\X_0\quad\mbox{\rm in }\B,\qquad \w(0)=\u_{s0}
\quad\mbox{\rm in }\B.
     \label{eq:ciX_linI}
\end{alignat}
\end{subequations}
\end{problem}
We set 
\begin{equation}
\label{eq:kernel}
\Kt=\{(\v,\z(t))\in\V\times\Hub: \c(\mmu,\vcXb(t)-\z(t))=0\ \forall\mmu\in\Hub\}.
\end{equation}
We observe that~\eqref{eq:compatibility} implies that
$(\u_0,\u_{s0})\in\mathbb{K}_0$.

Problem~\ref{pb:linI} is equivalent to the following one.
\begin{problem}
\label{pb:inK}
Given $(\u_0,\u_{s0})\in\K_0$, for almost every $t\in]0,T]$, find 
$(\u(t),\w(t))\in\Kt$ and $\X(t)\in\Hub$ such that
\begin{equation}
\label{eq:inK}
\aligned
&\rho_f\frac d {dt}(\u(t),\v)+a(\u(t),\v)
+\dr\left(\frac{\partial\w}{\partial t}(t),\z(t)\right)_{\B}
\\
&\qquad 
+\kappa(\Grad_s\X(t),\Grad_s\z(t))_{\B}=0&&\quad \forall(\v,\z(t))\in\Kt\\
&\left(\frac{\partial\X}{\partial t}(t),\y\right)_{\B}=(\w(t),\y)_{\B}
  &&\quad\forall\y\in\Ldb\\
&\u(0)=\u_0\quad\mbox{\rm in }\Omega,\qquad \w(0)=\u_{s0} \quad\mbox{\rm in }\B,
\\
&\X(0)=\X_0\quad\mbox{\rm in }\B,
\endaligned
\end{equation}
\end{problem}

In the following section we are going to prove existence and uniqueness of the
solution to Problem~\ref{pb:inK}. In the next section we will show
existence and existence for the pressure $p$ and the multipier $\lambda$ as
well.

\section{Existence and uniqueness}
\label{se:main}

We start this section by showing existence and uniqueness of the solution to
Problem~\ref{pb:inK} by following the Galerkin approximation technique used
in~\cite[Chapt.III.1]{T}. The proof of the next theorem will be obtained in
several steps.
\begin{thm}
\label{th:existence}
We set $\X_0(\s)=\s$ for $\s\in\B$.
Let $\Xb\in C^1([0,T];W^{1,\infty}(\B))$ be such that Assumption~\ref{ass} is
satisfied. Then, given $\u_0\in\V$ and $\u_{s0}\in\Hub$
satisfying the compatibility condition~\eqref{eq:compatibility}, for a.e.
$t\in(0,T)$ there exist $(\u(t),\w(t))\in\Kt$ and $\X(t)\in\Hub$ satisfying
Problem~\ref{pb:inK} and
\[
\aligned
&\u\in L^\infty(0,T;\H)\cap L^2(0,T;\V)\\
&\w\in L^\infty(0,T;\Ldb)\cap L^2(0,T;\Hub)\\
&\X\in L^\infty(0,T;\Hub)\quad \text{with } 
\frac{\partial\X}{\partial t}\in L^\infty(0,T;\Ldb)\cap L^2(0,T;\Hub).
\endaligned
\]
\end{thm}
\subsection{Basis in $\Kt$}
We introduce a basis in $\Kt$ that will be used for the Galerkin approximation
of our problem.
Let $\ppsi_j$ ($j\in\NA$) be the complete set of eigenfunctions for the
eigenvalue problem: find $\lf\in\RE$ and $\ppsi\in\V$ with $\ppsi\ne0$
such that
\begin{equation}
\label{eq:eigpsi}
a(\ppsi,\v)=\lf(\ppsi,\v)\qquad\forall \v\in\V.
\end{equation}
It is well known that the eigenvalues are positive and can be enumerated 
in an increasing sequence going to $+\infty$. The associated eigenfunctions
$\{\ppsi_j\}_{j=1}^\infty$ are orthogonal with respect to
the scalar product in $\mathbf{L}^2(\Omega)$ and to the bilinear form
$a(\cdot,\cdot)$. We normalize them with respect to the $\mathbf{L}^2(\Omega)$
norm, so that $\|\ppsi_j\|_{0,\Omega}=1$ for all $j\in\NA$.

Moreover, let $\cchi_j$ ($j\in\NA$) be the complete set of
eigenfunctions for the eigenvalue problem: find $\ls\in\RE$ and $\cchi\in\Hub$
with $\cchi\ne0$ such that
\begin{equation}
\label{eq:eigchi}
\c(\mmu,\cchi)=\ls(\cchi,\mmu)_{\B}\qquad\forall\mmu\in\Hub.
\end{equation}
Also the eigenvalues of~\eqref{eq:eigchi} are positive and can be enumerated
in increasing sequence going to $+\infty$.
We have that $\{\cchi_j\}_{j=1}^\infty$ are orthogonal we respect to
the scalar product in $\Ldb$ and to the bilinear form $\c(\cdot,\cdot)$. We
normalize them with respect to $\c$ so that $\c(\cchi_j,\cchi_j)=1$ for all
$j\in\NA$.

\begin{proposition}
\label{pr:base_phi}
For $j\in\NA$ and $t\in[0,T]$, let us set $\pphi_j(t)=\ppsi_j\circ\Xb(t)\in\Hub$.
Then, for each $t\in[0,T]$, $\{\pphi_j(t)\}_{j=1}^\infty$ is a basis of
$\Hub$.
\end{proposition}
\begin{proof}
Given $\z\in\Hub$ we will show that it can be written as a combination of
the $\{\pphi_j(t)\}$'s. Thanks to the assumptions on $\Xb$, we have that
$\v_z(t)=\z\circ\Xb(t)^{-1}\in\HH^1(\overline{\Os_t})$ where
$\overline{\Os_t}=\Xb(\B,t)$. Let $\tilde\v_z(t)\in\Huo$ be an extension of
$\v_z(t)$ to $\Omega$, so that $\tilde\v_z(t)|_{\overline{\Os_t}}=\v_z(t)$. Then
we can write $\tilde\v_z(t)$ in terms of the basis functions $\ppsi_j$, that is
\[
\tilde\v_z(t)=\sum_{j=1}^\infty\alpha_j(t)\ppsi_j.
\]
By construction we have that $\tilde\v_z(\Xb(\cdot,t),t)=\z\in\Hub$, hence we
obtain
\[
\z=\tilde\v_z(\Xb(\cdot,t),t)=\sum_{j=1}^\infty\alpha_j(t)\ppsi_j\circ\Xb
=\sum_{j=1}^\infty\alpha_j(t)\pphi_j(t).
\]

\end{proof} 
As a consequence of the previous proposition, a basis in $\Kt$ is given by
$\{(\ppsi_j,\pphi_j(t))\}_{j=1}^\infty\}$.

\subsection{Galerkin approximation}
We introduce a Galerkin approximation of the solution of Problem~\ref{pb:inK}.
Let us consider $\V^m=\Span(\ppsi_1,\dots,\ppsi_m)$,
$\W^m(t)=\Span(\pphi_1(t),\dots,\pphi_m(t))$, and
$\HH^m=\Span(\cchi_1,\dots,\cchi_m)$.
We define a subspace $\Kt^m$ of $\Kt$ generated by the first $m$
basis functions in $\Kt$ as follows
\begin{equation}
\label{eq:Km}
\Kt^m=\{(\v,\z(t))\in\V^m\times\W^m: \c(\pphi_i(t),\vcXb(t)-\z(t))=0\text{ for }
i=1,\dots,m\}
\end{equation}
It is clear that if $(\v,\z(t))\in\Kt^m$ then
\[
\v=\sum_{j=1}^m\alpha_j\ppsi_j\qquad \z(t)=\sum_{j=1}^m\alpha_j\pphi_j(t)
\]
for the same coefficients $\{\alpha_j\}$.
The Galerkin approximation of the solution of Problem~\ref{pb:inK} is given by
\begin{equation}
\label{eq:Galapprox}
\aligned
&\u^m(t)=\sum_{j=1}^m\alm_j(t)\ppsi_j,\quad
\w^m(t)=\sum_{j=1}^m\alm_j(t)\pphi_j(t), \quad\\
&\X^m(t)=\sum_{j=1}^m\bem_j(t)\cchi_j
\endaligned
\end{equation}
such that
\begin{equation}
\label{eq:inKm}
\aligned
&\rho_f\frac d {dt}(\u^m(t),\v)+a(\u^m(t),\v)
+\dr\left(\frac{\partial\w^m}{\partial t}(t),\z(t)\right)_{\B}\\
&\qquad 
+\kappa(\Grad_s\X^m(t),\Grad_s\z(t))_{\B}=0&& \forall(\v,\z(t))\in\Kt^m\\
&\left(\frac{\partial\X^m}{\partial t}(t),\y\right)_{\B}=(\w^m(t),\y)_{\B}
  &&\forall\y\in\HH^m\\
&\u^m(0)=\u^m_0\ \text{ in }\Omega,\qquad
\w^m(0)=\u^m_{s0} \ \text{ in }\B\\
&\X^m(\s,0)=\X^m_0\ \text{ for }\s\in\B.
\endaligned
\end{equation}
The initial conditions in~\eqref{eq:inKm} are obtained by projecting the
initial data, that is 
\begin{equation}
\label{eq:inKm_ini}
\aligned
&\u^m_0=\sum_{j=1}^m\alm_{0j}\ppsi_j\quad \text{with }
\alm_{0j}=\frac{(\u_0,\ppsi_j)}{(\ppsi_j,\ppsi_j)}\\
&\u^m_{s0}=\sum_{j=1}^m\alm_{0j}\pphi_j(0)\\
&\X^m_0=\sum_{j=1}^m\bem_{0j}\cchi_j\quad\text{with }
\bem_{0j}=\frac{(\s,\cchi_j)}{(\cchi_j,\cchi_j)},
\endaligned
\end{equation}
where we have taken into account the compatibility
assumption~\eqref{eq:compatibility}.

Using~\eqref{eq:Galapprox} in~\eqref{eq:inKm} we obtain for
$(\v,\z)=(\ppsi_i,\pphi_i(t))$ and $\y=\cchi_i$ the following system: 
\begin{equation}
\label{eq:systemm}
\aligned
&\rho_f\sum_{j=1}^m\alpha_j'(t)(\ppsi_j,\ppsi_i)
+\sum_{j=1}^m\alpha_j(t)a(\ppsi_j,\ppsi_i)\\
&\quad
+\dr\left(\sum_{j=1}^m\left(\alpha_j'(t)\pphi_j(t)
+\alpha_j(t)\pphi'_j(t)\right),\pphi_i(t)\right)_{\B}\\
&\quad
+\kappa\sum_{j=1}^m\beta_j(t)(\Grad_s\cchi_j,\Grad_s\pphi_i(t))_{\B}=0\\
& \sum_{j=1}^m\beta_j'(t)(\cchi_j,\cchi_i)_{\B}=
\sum_{j=1}^m\alpha_j(t)(\pphi_j(t),\cchi_i)_{\B},
\endaligned
\end{equation}
where we omitted the superscript $m$ in order to simplify the notation.

Thanks Proposition~\ref{pr:base_phi}, $\pphi_j(t)\in\Hub$ 
can be written in terms of the basis $\{\cchi_i\}_{i=1}^\infty$ as follows
\begin{equation}
\label{eq:phi_chi}
\pphi_j(t)=\sum_{r=1}^\infty\delta_{jr}(t)\cchi_r\qquad\text{with }
\delta_{jr}(t)=\c(\pphi_j(t),\cchi_r),
\end{equation}
therefore $\delta_{jr}(t)$ inherits the regularity in time of $\pphi_j(t)$. 
\begin{lemma}
\label{le:serie}
Under Assumption~\ref{ass}, we have for $j\in\NA$
\begin{equation}
\label{eq:serie}
\aligned
&\bigg\|\sum_{r=1}^\infty \delta_{jr}^2(t)c_r\bigg\|_{L^\infty(0,T)}\le C
\|\Xb\|^2_{L^\infty(L^\infty(\B))}\|\ppsi_j\|^2_\OO\\
&\bigg\|\sum_{r=1}^\infty \delta_{jr}^2(t)\bigg\|_{L^\infty(0,T)}
\le C
\|\Xb\|^2_{L^\infty(\W^{1,\infty}(\B))}\|\ppsi_j\|^2_{1,\Omega}
\\
&\bigg\|\sum_{r=1}^\infty (\delta_{jr}'(t))^2c_r\bigg\|_{L^\infty(0,T)}\le C
\|\Xb\|^2_{W^{1,\infty}(L^\infty(\B))}\|\ppsi_j\|^2_{1,\Omega},
\endaligned
\end{equation}
where $c_r=\|\cchi_r\|^2_{\OB}=\frac1{\lambda_{sr}}$ (see~\eqref{eq:eigchi}).
\end{lemma}
\begin{proof}
For each $j\in\NA$, $\ppsi_j\in\V$ is an eigenfunction of~\eqref{eq:eigpsi} with
$\|\ppsi_j\|^2_\OO=1$ and $\|\Grads{\ppsi_j}\|^2_\OO=\lambda_{fj}$.
Hence $\pphi_j$ is continuous from $[0,T]$ into $\Hub$ with the time
derivative in $L^\infty(0,T;\Ldb)$.
Taking into account the properties of the eigensolutions
of~\eqref{eq:eigchi} we have
\[
\|\pphi_j(t)\|^2_\OB=
\sum_{r=1}^\infty\delta_{jr}^2(t)\|\cchi_r\|^2_\OB
=\sum_{r=1}^\infty\delta_{jr}^2(t)c_r.
\]
Hence we have:
\[
\bigg\|\sum_{r=1}^\infty \delta_{jr}^2(t)c_r\bigg\|_{L^\infty(0,T)}=
\|\pphi_j(t)\|^2_{L^\infty(\Ldb)}
\le
C\|\Xb\|^2_{L^\infty(\L^\infty(\B))}\|\ppsi_j\|^2_{\OO}.
\]
Similarly, we set
$d_r=\|\Grad_s\cchi_r\|^2_\OB=\frac{\lambda_{sr}-1}{\lambda_{sr}}$ where
$\lambda_{sr}$ are the eigenvalues of~\eqref{eq:eigchi}. It is easy to see that
$0<\frac{\lambda_{s1}-1}{\lambda_{s1}}\le d_r<1$. Then we have
\[
\|\Grad_s\pphi_j(t)\|^2_\OB
=\sum_{r=1}^\infty\delta_{jr}^2(t)\|\Grad_s\cchi_r\|^2_\OB
=\sum_{r=1}^\infty\delta_{jr}^2(t)d_r.
\]
The above equations imply that for each $t\in[0,T]$ the series on the right
hand side is convergent and that the first two inequalities in~\eqref{eq:serie}
hold true. Let us now show the third one. We have that the time
derivative of $\pphi_j(t)$ is given by $\frac{\partial\pphi_j}{\partial t}(t)
=\Grad\ppsi_j\frac{\partial\Xb}{\partial t}(t)$, hence it belongs to
$L^\infty(0,T;\Ldb)$; moreover we have
\[
\frac{\partial\pphi_j}{\partial t}(t)=\sum_{r=1}^\infty \delta'_{jr}(t)\cchi_r,
\]
from which we obtain 
\[
\|\pphi'_j(t)\|^2_\OB=\sum_{r=1}^\infty (\delta'_{jr})^2c_r
\]
and we conclude again that the series on the right hand side is convergent and
that the estimate in the second inequality of~\eqref{eq:serie} is verified.
\end{proof}

Using the expression~\eqref{eq:phi_chi} into~\eqref{eq:systemm}, we arrive at
\[
\aligned
&\rho_f\alpha_i'(t)+a(\ppsi_i,\ppsi_i)\alpha_i(t)\\
&\qquad+\dr\sum_{j=1}^m\left(C_{ij}(t)\alpha_j'(t)+D_{ij}(t)\alpha_j(t)\right)
+\kappa\sum_{j=1}^m \delta_{ij}(t)d_j\beta_j(t)=0\\
&\beta_i'(t)=\sum_{j=1}^m\alpha_j(t)B_{ji}(t),
\endaligned
\]
where $B(t)$, $C(t)$, $D(t)$, and $E(t)$ are real matrices in
$\RE^{m\times m}$ with elements
\begin{equation}
\label{eq:matrici}
\aligned
&B_{ji}(t)=\delta_{ji}(t)&&\quad
C_{ij}(t)=\sum_{r=1}^\infty \delta_{jr}(t)\delta_{ir}(t)c_r\\
&D_{ij}(t)=\sum_{r=1}^\infty \delta'_{jr}(t)\delta_{ir}(t)c_r&&\quad
E_{ij}(t)=\delta_{ij}(t)d_j.
\endaligned
\end{equation}
Let $\aalm(t)$ and $\bbem(t)$ be the vector valued functions with components
$\alm_j(t)$ and $\bem_j(t)$, respectively.
We have obtained the following system of linear ordinary differential equations
\begin{equation}
\label{eq:ode}
\aligned
&\left(\rho_f\I_m+\dr C(t)\right)(\aalm(t))'
+\left(a(\ppsi_i,\ppsi_i)\I_m+D(t)\right)\aalm(t)\\
&\qquad+E(t)\bbem(t)=0\\
&(\bbem(t))'=B^T\aalm(t)\\
&\aalm(0)=\aalm_0\\
&\bbem(0)=\bbem_0,
\endaligned
\end{equation}
where $\aalm_0$ and $\bbem_0$ are the vectors with components $\alm_{0j}$
and $\bem_{0j}$ (see~\eqref{eq:inKm_ini}.
\begin{lemma}
\label{le:matrici}
Under Assumption~\ref{ass},
the matrices $B(t)$, $C(t)$, $D(t)$, and $E(t)$ given by ~\eqref{eq:matrici}
are well defined and continuous in $[0,T]$. Moreover, the matrix
$\rho_f\I_m+\dr C(t)$ is invertible with continuous inverse.
\end{lemma}
\begin{proof}
By definition~\eqref{eq:phi_chi}, it is clear that the elements of $B(t)$ and
$E(t)$ are continuous in $[0,T]$.
Let us consider the elements of $C(t)$. Thanks to the Cauchy--Schwarz inequality
we have
\[
\aligned
\|C_{ij}\|_{L^\infty(0,T)}
&=\left|\sum_{r=1}^\infty \delta_{jr}(t)\delta_{ir}(t)c_r\right|\\
&\le\left(\sum_{r=1}^\infty \delta_{jr}^2(t)c_r\right)^{1/2}
\left(\sum_{r=1}^\infty \delta_{ir}^2(t)c_r\right)^{1/2}\\
&\le\|\pphi_j\|_{L^\infty(\Ldb)}\|\pphi_i\|_{L^\infty(\Ldb)}.
\endaligned
\]
Since $\pphi_j(t)$ for $j\in\NA$ is continuous in $[0,T]$ with values in
$\Hub$, the series $\sum_{r=1}^\infty \delta_{jr}^2(t)c_r$ is continuous in
$[0,T]$. 
This implies that the elements of $C(t)$ are also continuous in $[0,T]$.

A similar argument shows that $\|D_{ij}\|_{L^\infty(0,T)}$ is bounded.

Now we show that $\rho_f\I_m+\dr C(t)$ is invertible.
Since $C(t)$ is symmetric, it is enough to show that $C$ is also positive
semidefinite that is $x^TCx\ge0$ for all $x\in\RE^m$. This can be obtained by
direct computation as follows
\[
\aligned
x^TCx&=
\sum_{i,j=1}^mx_i\left(\sum_{r=1}^\infty \delta_{ir}(t)c_r\delta_{jr}(t)\right)x_j
\\
&=\sum_{r=1}^\infty c_r \left(\sum_{i=1}^mx_i\delta_{ir}(t)\right)
\left(\sum_{j=1}^mx_j\delta_{jr}(t)\right)\\
&=\sum_{r=1}^\infty c_r \left(\sum_{i=1}^mx_i\delta_{ir}(t)\right)^2\ge0.
\endaligned
\]
%}
\end{proof}
\begin{proposition}
\label{pr:ode_existence}
The system of ordinary differential
equations~\eqref{eq:ode} has a unique solution $\aalm\in C^1([0,T])$ and
$\bbem\in C^1([0,T])$.
\end{proposition}
\begin{proof}
As a consequence of Lemma~\ref{le:matrici}, the matrix $\rho_f\I_m+\dr C(t)$ 
is invertible with continuous inverse, hence the standard theory for systems
of linear first order ordinary differential equations gives
that~\eqref{eq:ode} has a unique solution in $C^1([0,T])$.
\end{proof}
The above proposition yields the existence of the solution of~\eqref{eq:inKm},
stated in the following theorem
\begin{thm}
\label{th:gal_existence}
There exists a unique solution $(\u^m(t),\w^m(t))\in\Kt^m$ and $\X^m(t)\in\HH^m$
of~\eqref{eq:inKm} and~\eqref{eq:inKm_ini} with
\begin{equation}
\label{eq:Galregu}
(\u^m,\w^m)\in C^1([0,T];\Kt^m), \quad \X^m\in C^1([0,T];\HH^m).
\end{equation}
\end{thm}
\subsection{A priori estimates}
We have the following a priori estimates for the solution of~\eqref{eq:inKm}
and~\eqref{eq:inKm_ini}. 
\begin{proposition}
\label{pr:apriori}
The following bounds hold true with $C>0$ independent of $m$:
\begin{subequations}
\begin{alignat}{1}
&\|\u^m\|_{L^\infty(\Ld)}+\|\u^m\|_{L^2(\Huo)}\le
C\left(\|\u^m_0\|_\OO+\|\u^m_{s0}\|_\OB+|\B|^{1/2}\right)
\label{eq:aprioriu}\\
&\|\w^m\|_{L^\infty(\Ldb)}
+\|\w^m\|_{L^2(\Hub)}\le
C\left(\|\u^m_0\|_\OO+\|\u^m_{s0}\|_\OB+|\B|^{1/2}\right)
\label{eq:aprioriw}\\
&\|\X^m\|_{L^\infty(\Hub)}\le
C\left(\|\u^m_0\|_\OO+\|\u^m_{s0}\|_\OB+|\B|^{1/2}\right)
\label{eq:aprioriX}\\
&\left\|\frac{\partial\X^m}{\partial t}\right\|_{L^\infty(\Ldb)}
+\left\|\frac{\partial\X^m}{\partial t}\right\|_{L^\infty(\Hub)}\le
C\big(\|\u^m_0\|_\OO+\|\u^m_{s0}\|_\OB+|\B|^{1/2}\big),
\label{eq:aprioriXt}
\end{alignat}
\end{subequations}
where $|\B|$ stands for the measure of $\B$.
\end{proposition}
\begin{proof}
By definition~\eqref{eq:Galapprox} we have that
\[
\frac{\partial\u^m}{\partial t}\in L^2(0,T;\V)\qquad
\frac{\partial\w^m}{\partial t}\in L^2(0,T;\Hub)\qquad
\frac{\partial\X^m}{\partial t}\in L^2(0,T;\Hub)
\]
implying that
\[
\aligned
&2\left(\frac{\partial\u^m}{\partial t}(t),\u^m(t)\right)
=\frac{d}{dt}\|\u^m(t)\|^2_\OO\\
&2\left(\frac{\partial\w^m}{\partial t}(t),\w^m(t)\right)_{\B}
=\frac{d}{dt}\|\w^m(t)\|^2_\OB\\
&2\left(\Grad_s\frac{\partial\X^m}{\partial t}(t),\Grad_s\X^m(t)\right)_{\B}
=\frac{d}{dt}\|\Grad_s\X^m(t)\|^2_\OB.
\endaligned
\]
Let us take $(\v,\z(t))=(\u^m(t),\w^m(t))$ in the first equation
in~\eqref{eq:inKm}, then
\[
\aligned
&\frac{\rho_f}2\frac{d}{dt}\|\u^m(t)\|^2_\OO+\nu_0\|\Grads{\u^m(t)}\|^2_\OO\\
&\qquad
+\frac{\dr}2\frac{d}{dt}\|\w^m(t)\|^2_\OB
+\kappa(\Grad_s\X^m(t),\Grad_s\w^m(t))_{\B}\le0.
\endaligned
\]
Thanks to the fact that $(\u^m(t),\w^m(t))\in\Kt^m$, we have that
$\w^m(t)$ belongs to $\Hub$ and the second equation in~\eqref{eq:inKm} implies
that it is equal to the time derivative of $\X^m(t)$, so that the last
inequality can be rewritten as
\[
\aligned
&\frac{\rho_f}2\frac{d}{dt}\|\u^m(t)\|^2_\OO+\nu_0\|\Grads{\u^m(t)}\|^2_\OO\\
&\qquad
+\frac{\dr}2\frac{d}{dt}\|\w^m(t)\|^2_\OB
+\frac{\kappa}2\frac{d}{dt}\|\Grad_s\X^m(t)\|^2_\OB\le0.
\endaligned
\]
Integrating on $(0,t)$ with $t\in(0,T]$ and taking into
account~\eqref{eq:Korn}, we arrive at
\begin{equation}
\label{eq:final}
\aligned
\rho_f\|\u^m(t)\|^2_\OO&+2\mathbf{k}\int_0^t\|\Grad\u^m(\tau)\|^2_\OO\,d\tau%\\
%&\qquad
+\dr\|\w^m(t)\|^2_\OB+\kappa\|\Grad_s\X^m(t)\|^2_\OB\\
&\le\rho_f\|\u^m_0\|^2_\OO+\dr\|\u^m_{s0}\|^2_\OB+\kappa\|\X^m_0\|^2_\OB.
\endaligned
\end{equation}
Thanks to~\eqref{eq:inKm_ini}, the last inequality
implies~\eqref{eq:aprioriu}. In order to
obtain~\eqref{eq:aprioriw}, we observe that a.e. in $t$
$\w^m(t)=(\u^m\circ\Xb)(t)$ in $\Hub$. Therefore 
\[
\Grad_s\w^m(t)=(\Grad\u^m\circ\Xb)(t)\Grad_s\Xb(t)
\]
and
\[
\aligned
\|\Grad_s\w^m\|^2_{L^2(\Ldb)}&=\int_0^T\|\Grad_s\w^m(t)\|^2_\OB\,dt\\
&=\int_0^T\|(\Grad\u^m\circ\Xb)(t)\Grad_s\Xb(t)\|^2_\OB\,dt\\
&\le \|\Grad_s\Xb(t)\|^2_{L^\infty(\L^\infty(\B))}
\int_0^T\|(\Grad\u^m\circ\Xb)(t)\|^2_\OB\,dt\\
&\le C\int_0^T\|\Grad\u^m\|^2_{0,\Os_t}\,dt\le C\|\Grad\u^m\|^2_{L^2(\Ld)}
\endaligned
\]
which together with~\eqref{eq:final} gives~\eqref{eq:aprioriw}.
It remains to bound $\X^m$. Since
\[
\frac{\partial\X^m(t)}{\partial t}=\w^m(t),
\]
we obtain~\eqref{eq:aprioriXt} directly.
Moreover, the inequality~\eqref{eq:final} gives the estimate for
$\|\Grad_s\X^m\|_{L^\infty(\Ldb)}$. Let us now estimate $\X^m(t)$ in the
$\Ld$-norm. Integrating in time the last equation, we obtain
\begin{equation}
\label{eq:Xm}
\X^m(t)=\X_0^m+\int_0^t\w^m(\tau)d\tau.
\end{equation}
After some computations, we get
\[
\aligned
\|\X^m(t)\|^2_\OB&\le\|\X_0^m\|^2_\OB
+\left\|\int_0^t\w^m(\tau)d\tau\right\|^2_\OB\\
&\le\|\X_0^m\|^2_\OB+\left\|t\int_0^t|\w^m(\tau)|^2d\tau\right\|^2_\OB\\
&\le\|\X_0^m\|^2_\OB+t\int_0^t\|\w^m(\tau)\|^2_\OB d\tau.
\endaligned
\]
It follows
\[
\aligned
&\|\X^m\|^2_{L^2(\Ldb)}\le\|\X_0^m\|^2_\OB+T\|\w^m\|^2_{L^2(\Ldb)}\\
&\|\X^m\|^2_{L^\infty(\Ldb)}\le\|\X_0^m\|^2_\OB
+T^2\|\w^m\|^2_{L^\infty(\Ldb)}.
\endaligned
\]
\end{proof}
\subsection{Passing to the limit}\leavevmode

\medskip
\noindent
{\bf Step 1}: $\u^m$ converges to $\u\in L^\infty(0,T;\H)\cap L^2(0,T;\V)$.\\
The a priori estimate~\eqref{eq:aprioriu} shows the existence of an element
$\u\in L^\infty(0,T;\H)$ and of a subsequence $m'\to\infty$ such that
\[
\u^{m'}\weakstar \u\quad\text{in }L^\infty(0,T;\H).
\]
This means that for each $\v\in L^1(0,T;\H)$
\begin{equation}
\label{eq:ws}
\int_0^T(\u^{m'}(t)-\u(t),\v(t))\,dt\to 0\qquad\text{as } m'\to\infty.
\end{equation}
Since $\u^{m'}$ is also bounded in $L^2(0,T;\V)$, we can extract another
subsequence (still denoted $\u^{m'}$) that converges weakly to
$\u^*\in L^2(0,T;\V)$, that is
\[
\u^{m'}\rightharpoonup\u^* \quad\text{in }L^2(0,T;\V).
\]
The above convergence means that 
\begin{equation}
\label{eq:weak}
\int_0^T \langle\u^{m'}(t)-\u^*(t),\v(t)\rangle\,dt\to 0\qquad\text{as }
m'\to\infty\quad\forall\v\in L^2(0,T;\V').
\end{equation}
By the Riesz representation theorem, we can identify $\H$ with $\H'$, so that
\[
\V\subset\H=\H'\subset\V'.
\]
Moreover the duality pairing between $\V'$ and $\V$ can be identified to the
scalar product in $\H$ for $\u\in\V$ and $\v\in\H$, that is
\[
\langle\v,\u\rangle=(\v,\u)\quad\forall\v\in\H,\ \forall\u\in\V.
\]
Comparing~\eqref{eq:ws} and~\eqref{eq:weak} with $\v\in L^2(0,T;\H)$,
we obtain that
\begin{equation}
\label{eq:solu}
\u=\u^*\in L^\infty(0,T;\H)\cap L^2(0,T;\V).
\end{equation}
{\bf Step 2}: $\w^m$ converges to $\w$ in
$L^\infty(0,T;\Ldb)\cap L^2(0,T;\Hub)$.\\
With arguments similar to those used above, we obtain from~\eqref{eq:aprioriw}
the following convergence for some subsequences of $\w^{m'}$:
\begin{equation}
\label{eq:convw}
\aligned
&\w^{m'}\weakstar \w&&\text{in }L^\infty(0,T;\Ldb)\\
&\w^{m'}\rightharpoonup \w^*&&\text{in }L^2(0,T;\Hub).
\endaligned
\end{equation}
Using again the Riesz representation theorem, we have that
$\Hub\subset\Ldb\subset\Hub'$ and we can conclude that
\begin{equation}
\label{eq:solw}
\w=\w^* \in L^\infty(0,T;\Ldb)\cap L^2(0,T;\Hub).
\end{equation}
{\bf Step 3}: The limit $(\u(t),\w(t))$ is contained in $\Kt$.\\
By construction $(\u^m(t),\w^m(t))\in\Kt^m$, that is
\[
\c(\pphi_i(t),(\u^m\circ\Xb)(t)-\w^m(t))=0\quad \text{for }i=1,\dots,m.
\]
Since $\u^m\rightharpoonup\u$ in $L^2(0,T;\V)$, we have that
$\u^m\circ\Xb\rightharpoonup\u\circ\Xb$ in $L^2(0,T;\Hub)$.

Let us consider a scalar function $\phi\in C^\infty(0,T)$. Then we have that
$\phi\pphi_i$ belongs to $L^2(0,T;\Hub)$ and that
\[
\aligned
0&=\int_0^T\phi(t)\c(\pphi_i(t),(\u^m\circ\Xb)(t)-\w^m(t))\,dt\\
&=\int_0^T\c(\phi(t)\pphi_i(t),(\u^m\circ\Xb)(t)-\w^m(t))\,dt.
\endaligned
\]
Recalling that the bilinear form $\c(\cdot,\cdot)$ is the scalar product in
$\Hub$, we can pass to the limit as $m\to\infty$. The weak convergence of
$\u^m\circ\Xb$ and of $\w^m$ in $L^2(0,T;\Hub)$ implies
\begin{equation}
\label{eq:cuno}
\int_0^T\phi(t)\c(\pphi_i,\ucXb-\w(t))\,dt=0
\quad\text{for }i=1,\dots,m.
\end{equation}
The last equality is valid for each $i$, and by linearity for all
finite linear combinations of $\pphi_i(t)$. Using
Proposition~\ref{pr:base_phi}, by continuity, Equation~\eqref{eq:cuno} is
still valid for all $\mmu\in\Hub$ and implies that 
\[
\c(\mmu,\ucXb-\w(t))=0\quad \forall\mmu\in\Hub
\]
holds true in the sense of distributions on $(0,T)$, so that we conclude that
$(\u(t),\w(t))$ belongs to $\Kt$.

\medskip\noindent
{\bf Step 4}: Limit of $\X^m$ and ${\partial\X^m}/{\partial t}$.\\
Since $\X^m$ and $\|{\partial\X^m}/{\partial t}\|$ are bounded in
$L^\infty(0,T;\Hub)$ and $L^\infty(0,T;\Ldb)$, respectively, there exist
$\X\in L^\infty(0,T;\Hub)$, $\Y\in L^\infty(0,T;\Ldb)$, and subsequences $m'$
such that
\begin{equation}
\label{eq:convX}
\aligned
&\X^{m'}\weakstar\X&&\text{in }L^\infty(0,T;\Hub)\\
&\frac{\partial\X^{m'}}{\partial t}\weakstar \Y&&
\text{in }L^\infty(0,T;\Ldb)
\endaligned
\end{equation}
in the sense that
\[
\aligned
&\int_0^T\langle\X^{m'}(t)-\X(t),\y(t)\rangle_{\B}\,dt\to 0
&&\text{as }m\to\infty\quad\forall\y\in L^1(0,T;\Hub')\\
&\int_0^T\left(\frac{\partial\X^{m'}}{\partial t}(t)-\Y(t),\y(t)\right)_{\B}\to0\ 
&&\text{as }m\to\infty \quad\forall\y\in L^1(0,T;\Ldb).
\endaligned
\]
Let us consider a scalar function $\phi(t)$ which is continuously
differentiable in $[0,T]$ and $\phi(T)=0$, and let us denote by $\phi'(t)$ its
derivative. Then for $j=1,\dots,m$
\[
\int_0^T\left(\frac{\partial\X^m}{\partial t}(t),\cchi_j\right)_{\B}\phi(t)\,dt
=-\int_0^T\left(\X^m(t),\phi'(t)\cchi_j\right)_{\B}\,dt
-(\X_0^m,\cchi_j)_{\B}\phi(0).
\]
From~\eqref{eq:inKm_ini} we have that $\X_0^m\to\X_0$ strongly in $\Hub$, hence
we can pass to the limit and obtain
\[
\int_0^T\left(\Y(t),\cchi_j\right)_{\B}\phi(t)\,dt
=-\int_0^T\left(\X(t),\phi'(t)\cchi_j\right)_{\B}\,dt-(\X_0,\cchi_j)_{\B}\phi(0).
\]
The above relation is valid for all finite linear combinations $\y$ of
$\cchi_j$ with $j=1,\dots,m$. Moreover, it depends linearly and continuously
on $\y\in\Ldb$; hence, it is valid for all $\y\in\Ldb$.
Taking $\phi\in{\mathscr D}(0,T)$ and integrating by parts, 
we get the following equation in the sense of distributions:
\[
\left(\frac{\partial\X}{\partial t}(t),\y\right)_{\B}=(\Y(t),\y)_{\B}
\quad\forall\y\in\Ldb.
\]
{\bf Step 5}: ${\partial\X}/{\partial t}(t)=\w(t)$.\\
We have (see~\eqref{eq:inKm})
\[
\left(\frac{\partial\X^m}{\partial t}(t)-\w^m(t),\cchi_i\right)_{\B}=0
\qquad\forall\cchi_i\ i=1,\dots,m.
\]
The convergence of $\w^m$ and of $\X^m$ obtained in~\eqref{eq:convw}
and~\eqref{eq:convX} implies that
\[
\Y=\w \in L^\infty(0,T;\Ldb),
\]
therefore the limits $\X$ and $\w$ satisfy the second equation
in~\eqref{eq:inK}.

\noindent
{\bf Step 6}: Passing to the limit in Equation~\eqref{eq:inKm}.\\
Let $\phi(t)$ be defined as before. We have:
\[
\int_0^T\left(\frac{\partial\u^m}{\partial t},\ppsi_j\right)\phi(t)\,dt=
-\int_0^T(\u^m(t),\ppsi_j\phi'(t))\,dt-(\u^m(0),\ppsi_j)\phi(0)
\]
and
\[
\aligned
&\int_0^T\left(\frac{\partial\w^m}{\partial t},\pphi_j(t)\right)_{\B}\phi(t)\,dt=
-\int_0^T(\w^m(t),\pphi_j(t)\phi'(t))_{\B}\,dt\\
&\qquad-\int_0^T(\w^m(t),\pphi_j'(t)\phi(t))_{\B}\,dt
-(\w^m(0),\pphi_j(0))_{\B}\phi(0).
\endaligned
\]
Using these relations in~\eqref{eq:inKm} we obtain
\[
\aligned
&-\rho_f\int_0^T(\u^m(t),\ppsi_j\phi'(t))\,dt+\int_0^Ta(\u^m(t),\ppsi_j\phi(t))\,dt\\
&\qquad
-\dr\int_0^T(\w^m(t),\pphi_j(t)\phi'(t))_{\B}\,dt
-\dr\int_0^T(\w^m(t),\frac{\partial\pphi_j}{\partial t}(t)\phi(t))_{\B}\,dt\\
&\qquad
+\kappa\int_0^T(\Grad_s\X^m(t),\Grad_s\pphi_j(t)\phi(t))_{\B}\,dt\\
&\quad=\rho_f(\u^m(0),\ppsi_j)\phi(0)+\dr(\w^m(0),\pphi_j(0))_{\B}\phi(0).
\endaligned
\]
For $j$ fixed, passing to the limit yields
\begin{equation}
\label{eq:i}
\aligned
&-\rho_f\int_0^T(\u(t),\ppsi_j\phi'(t))\,dt+\int_0^Ta(\u(t),\ppsi_j\phi(t))\,dt\\
&\qquad
-\dr\int_0^T(\w(t),\pphi_j(t)\phi'(t))_{\B}\,dt
-\dr\int_0^T(\w(t),\frac{\partial\pphi_j}{\partial t}(t)\phi(t))_{\B}\,dt\\
&\qquad
+\kappa\int_0^T(\Grad_s\X(t),\Grad_s\pphi_j(t)\phi(t))_{\B}\,dt\\
&\quad=\rho_f(\u_0,\ppsi_j)\phi(0)+\dr(\u_{s0},\pphi_j(0))_{\B}\phi(0).
\endaligned
\end{equation}
Each element $(\v,\z(t))\in\Kt^m$ can be written as 
\[
\aligned
&\v=\sum_{j=1}^m a_j\ppsi_j\\
&\z(t)=\sum_{j=1}^m a_j\pphi_j(t)=\sum_{j=1}^m a_j\ppsi_j\circ\Xb(t).
\endaligned
\]
Let us denote by $\z'(t)$ the time derivative of $\z(t)$. We have
\[
\z'(t)=\sum_{j=1}^m a_j\Grad\ppsi_j\circ\Xb(t)\frac{\partial\Xb}{\partial t}(t)
\]
which, due to the regularity of $\Xb$, is continuous from $[0,T]$ into $\Ldb$.

We write~\eqref{eq:i} as follows
\begin{equation}
\label{eq:ognivz}
\aligned
&-\rho_f\int_0^T(\u(t),\v\phi'(t))\,dt+\int_0^Ta(\u(t),\v\phi(t))\,dt\\
&\qquad
-\dr\int_0^T(\w(t),\z(t)\phi'(t))_{\B}\,dt
-\dr\int_0^T(\w(t),\z'(t)\phi(t))_{\B}\,dt\\
&\qquad
+\kappa\int_0^T(\Grad_s\X(t),\Grad_s\z(t)\phi(t))_{\B}\,dt\\
&\quad=\rho_f(\u_0,\v)\phi(0)+\dr(\u_{s0},\z(0))_{\B}\phi(0).
\endaligned
\end{equation}
All the terms depend linearly and continuously on $(\v,\z)\in C^1([0,T];\Kt^m)$,
hence for each $t\in[0,T]$ Equation~\eqref{eq:ognivz} holds true for all
$(\v,\z)\in C^1([0,T];\Kt)$.

Taking $\phi\in{\mathscr D}(0,T)$ and integrating by parts with respect
to $t$, we arrive at
\[
\aligned
&\rho_f\int_0^T\left(\frac{\partial\u}{\partial t}(t),\v\right)\phi(t)\,dt
+\int_0^Ta(\u(t),\v)\phi(t)\,dt\\
&\qquad
+\dr\int_0^T\left(\frac{\partial\w}{\partial t}(t),\z(t)\right)_{\B}\phi(t)\,dt
+\kappa\int_0^T(\Grad_s\X(t),\Grad_s\z(t))_{\B}\phi(t)\,dt=0,
\endaligned
\]
which implies that the first equation in~\eqref{eq:inK} holds true in the sense
of distributions on $(0,T)$. 

\medskip\noindent
{\bf Step 7}: Initial conditions.\\
It remains to check that $\u(0)=\u_0$, $\w(0)=\u_{s,0}$ and
$\X(0)=\X_0$.

Since $\frac{\partial\X}{\partial t}=\w\in L^2(0,T;\Hub)$, then $\X$ is
continuous from $[0,T]$ to $\Hub$, and we can pass to the limit
in~\eqref{eq:Xm} for $t=0$ arriving at $\X(0)=\X_0$.

We recall that
\[
(\u,\w)\in L^2(0,T;\V\times\Hub).
\]
Moreover, since $(\u(t),\w(t))\in\Kt$ and $\Kt\subset\V\times\Hub$, we have
that
\[
(\u,\w)\in L^2(0,T;\Kt). 
\]

In order to prove the continuity of $\u$ and $\w$ in the correct spaces
for the initial conditions, we can use a general interpolation theorem.
If we can show
\begin{equation}
\left(\frac{\partial\u}{\partial t},\frac{\partial\w}{\partial t}\right)\in
L^2(0,T;\V'\times\Hub')
\label{eq:HB}
\end{equation}
then it follows
\[
(\u,\w)\in C^0([0,T];\H\times\Ldb).
\]

Given $(\v,\z(t))\in\Kt$ the first equation in~\eqref{eq:inK} can be written
as
\[
\rho_f\left(\frac{\partial\u}{\partial t}(t),\v\right)+
\dr\left(\frac{\partial\w}{\partial t}(t),\z(t)\right)_{\B}
=-\langle A\u(t),\v\rangle-\kappa(\Grad_s\X(t),\Grad_s(\z(t)))_{\B},
\]
where $A:\V\to \V'$ is the linear and continuous operator associated with the
bilinear form $a$.
Since $\u\in L^2(0,T;\V)$, the function $A\u$ belongs to $L^2(0,T;\V')$.
Taking into account that $\X\in L^2(0,T;\Hub)$ we obtain also
\begin{equation}
\label{eq:secondo}
\int_0^T(\Grad_s\X(t),\Grad_s(\vcXb(t)))_{\B}\,dt\le
C \|\Grad_s\X(t)\|_{L^2(\Ldb)}\|\v\|_{\V}.
\end{equation}
It follows that
\[
\left(\frac{\partial\u}{\partial t},\frac{\partial\w}{\partial t}\right)\in
L^2(0,T;\Kt'),
\]
which, from Hahn--Banach theorem implies~\eqref{eq:HB}.

The general interpolation theory of Lions--Magenes~\cite{LM}
and~\cite[Lemma~1.2]{T} implies that $\u$ is continuous form $[0,T]$ to $\H$
and $\w$ is continuous from $[0,T]$ to $\Ldb$.  

It remains to check that $\u(0)=\u_0$ and $\w(0)=\u_{s0}$.
We multiply the first equation in~\eqref{eq:inK} by a scalar function
$\phi(t)$, continuously differentiable on $[0,T]$ with $\phi(T)=0$, and
integrate with respect to $t$
\[
\aligned
&\int_0^T\rho_f\frac d {dt}(\u(t),\v)\phi(t)\,dt 
+\int_0^T\dr\left(\frac{\partial\w}{\partial t}(t),\z(t)\right)_{\B}\phi(t)\,dt
\\
&+\int_0^Ta(\u(t),\v)\phi(t)\,dt
+\int_0^T\kappa(\Grad_s\X(t),\Grad_s\z(t))_{\B}\phi(t)\,dt=0.
\endaligned
\]
Integration by parts in the first two integrals gives for $(\v,\z(t))\in\Kt$
\[
\aligned
&\rho_f\int_0^T\frac d {dt}(\u(t),\v)\phi(t)\,dt 
+\dr\int_0^T
\left(\frac{\partial\w}{\partial t}(t),\z(t)\right)_{\B}\phi(t)\,dt\\
&\quad =-\rho_f\int_0^T(\u(t),\v)\phi'(t)\,dt
-\dr\int_0^T\left(\w(t),\z(t)\phi'(t)
+\frac{\partial\z}{\partial t}(t)\phi(t)\right)_{\B}\,dt\\
&\qquad -\rho_f(\u(0),\v)\phi(0)-\dr(\w(0),\z(0))_{\B}\phi(0),
\endaligned
\]
which inserted in the previous equation gives
\[
\aligned
&-\rho_f\int_0^T(\u(t),\v)\phi'(t)\,dt
-\dr\int_0^T\left(\w(t),\z(t)\phi'(t)
+\frac{\partial\z}{\partial t}(t)\phi(t)\right)_{\B}\,dt\\
&\qquad
+\int_0^Ta(\u(t),\v)\phi(t)\,dt
+\int_0^T\kappa(\Grad_s\X(t),\Grad_s\z(t))_{\B}\phi(t)\,dt\\
&\quad
=\rho_f(\u(0),\v)\phi(0)+\dr(\w(0),\z(0))_{\B}\phi(0).
\endaligned
\]
Comparing the last equality with~\eqref{eq:ognivz} gives for $\phi(0)\ne0$
\begin{equation}
\label{eq:initialcond}
\rho_f(\u(0)-\u_0,\v)+\dr(\w(0)-\u_{s0},\z(0))_{\B}=0
\end{equation}
for all $(\v,\z(0))\in\mathbb{K}_0$. Recalling that $\B=\Os_0$,
$\w(0)=\ucXb(0)$, and the
compatibility assumption~\eqref{eq:compatibility}, we can write the last
equation as
\[
\rho_f(\u(0)-\u_0,\v)_{\Of_0}+\rho_s(\u(0)-\u_{0},\v)_{\Os_0}=0,
\]
that is, with obvious notation,
%setting $\rho=\rho_f$ in $\Of_t$ and $\rho=\rho_s$ in $\Os_t$, 
$(\rho(\u(0)-\u_0),\v)_{\Omega}=0$ 
which gives that $\u(0)=\u_0$ in
$\Omega$. Substituting in~\eqref{eq:initialcond} we obtain also that
$\w(0)=\u_{s0}$ in $\B$.

\subsection{Uniqueness}
Let us assume that $(\u_1,\w_1,\X_1)$ and $(\u_2,\w_2,\X_2)$ are two solutions
of~\eqref{eq:inK}. Since the problem is linear the differences
$\hat\u=\u_1-\u_2$, $\hat\w=\w_1-\w_2$ and $\hat\X=\X^1-\X^2$ satisfy the
same equations as $\u,\w,\X$ with vanishing initial conditions, that is
\[
\aligned
&\rho_f\left(\frac{\partial\hat\u}{\partial t}(t),\v\right)+a(\hat\u(t),\v)
+\dr\left(\frac{\partial\hat\w}{\partial t}(t),\z(t)\right)_{\B}
\\
&\qquad 
+\kappa(\Grad_s\hat\X(t),\Grad_s\z(t))_{\B}=0&&\quad \forall(\v,\z(t))\in\Kt\\
&\left(\frac{\partial\hat\X}{\partial t}(t),\y\right)_{\B}=(\hat\w(t),\y)_{\B}
  &&\quad\forall\y\in\Ldb\\
&\hat\u(0)=0\quad\mbox{\rm in }\Omega,\quad \hat\w(0)=0 \quad\mbox{\rm in }\B,
\quad
\hat\X(0)=0\quad\mbox{\rm in }\B,
\endaligned
\]
We take $(\v,\z(t))=(\hat\u(t),\hat\w(t))$ in the first equation and use the
fact that $\hat\w(t)=\frac{\partial\hat\X}{\partial t}(t)$, so that we get
\begin{equation}
\label{eq:uniq1}
\aligned
&\rho_f\left(\frac{\partial\hat\u}{\partial t}(t),\hat\u(t)\right)
+a(\hat\u(t),\hat\u(t))
+\dr\left(\frac{\partial\hat\w}{\partial t}(t),\hat\w(t)\right)_{\B}
\\
&\qquad 
+\kappa
\left(\Grad_s\hat\X(t),\Grad_s\frac{\partial\hat\X}{\partial t}(t)\right)_{\B}=0.
\endaligned
\end{equation}
Thanks to~\eqref{eq:solu}, \eqref{eq:solw}, \eqref{eq:convX}, \eqref{eq:HB},
Step 5, and~\cite[Lemma III.1.2]{T}, we can write
\[
\aligned
&\left(\frac{\partial\hat\u}{\partial t}(t),\hat\u(t)\right)
=\frac12\frac{d}{dt}\|\hat\u(t)\|^2_\OO\\
&\left(\frac{\partial\hat\w}{\partial t}(t),\hat\w(t)\right)_{\B}
=\frac12\frac{d}{dt}\|\hat\w(t)\|^2_\OB\\
&\left(\Grad_s\hat\X(t),\Grad_s\frac{\partial\hat\X}{\partial t}(t)\right)_{\B}
=\frac12\frac{d}{dt}\|\Grad_s\hat\X(t)\|^2_\OB.
\endaligned
\]
Inserting these equalities into~\eqref{eq:uniq1} gives
\[
\frac{\rho_f}2\frac{d}{dt}\|\hat\u(t)\|^2_\OO+\mathbf{k}\|\Grad\hat\u(t)\|^2_\OO
+\frac{\dr}2\frac{d}{dt}\|\hat\w(t)\|^2_\OB
+\frac{\kappa}2\frac{d}{dt}\|\Grad_s\hat\X(t)\|^2_\OB\le0,
\]
which integrated from $0$ to $t$ implies
\[
\frac{\rho_f}2\|\hat\u(t)\|^2_\OO
+\frac{\dr}2\|\hat\w(t)\|^2_\OB
+\frac{\kappa}2\|\Grad_s\hat\X(t)\|^2_\OB\le0.
\]
Therefore $\u_1(t)=\u_2(t)$, $\w_1(t)=\w_2(t)$, and $\X_1(t)=\X_2(t)$ for all
$t$.
\section{Recovery of the pressure and of the Lagrange multiplier}
\label{se:pl}
In order to obtain existence and uniqueness of the solution of
Problem~\ref{pb:linI}, we need to show that starting from the solution
$(\u,\w,\X)$ of Problem~\ref{pb:inK}, we can define a Lagrange multiplier
$\llambda$ and a pressure $p$ so that $(\u,p,\X,\w,\llambda)$
satisfies~\eqref{eq:NS1_linI}-\eqref{eq:ciX_linI}.

\begin{proposition}
\label{pr:lambda}
Let $(\u,\w,\X)$ be the solution of Problem~\ref{pb:inK}, then there exists
$\llambda\in L^2(0,T;\Hub)$ such that for all $t\in(0,T)$
\begin{equation}
\label{eq:lambda}
\c(\llambda(t),\z)=
\dr\left(\frac{\partial\w}{\partial t}(t),\z\right)_{\B}
  +\kappa(\Grad_s\X(t),\Grad_s\z)_{\B}\qquad\forall\z\in\Hub.
\end{equation}

\end{proposition}
\begin{proof}
Since $\c$ is equal to the scalar product in $\Hub$, it is enough to show that
the right hand side is a linear functional on $\Hub$. The linearity is obvious.
We check now that both terms are continuous.
Since $\frac{\partial\w}{\partial t}\in L^2(0,T;\Hub')$ and
$\X\in L^2(0,T;\Hub)$ we have
\[
\aligned
&\int_0^T\left(\frac{\partial\w}{\partial t}(t),\z\right)_{\B}\,dt\le 
\left\|\frac{\partial\w}{\partial t}\right\|_{L^2(\Hub')}\|\z\|_{\Hub}
\\
&\int_0^T (\Grad_s\X(t),\Grad_s\z)_{\B}\,dt\le\|\X\|_{L^2(\Hub)}\|\z\|_{\Hub}.
\endaligned
\]
These inequalities  imply that
the right hand side of~\eqref{eq:lambda} is a continuous functional on
$L^2(0,T;\Hub)$.
Therefore, from the Lax--Milgram lemma, we obtain existence and uniqueness of
the solution $\llambda\in L^2(0,T;\Hub)$.
\end{proof}
The above proposition allows us to split the first equation in~\eqref{eq:inK}
into two equations as follows:
\begin{equation}
\label{eq:split}
\aligned
&\rho_f\frac d {dt}(\u(t),\v)+a(\u(t),\v)+\c(\llambda(t),\vcXb(t))=0
&&\quad\forall\v\in\V\\
&\dr\left(\frac{\partial\w}{\partial t}(t),\z\right)_{\B}
+\kappa(\Grad_s\X(t),\Grad_s\z)_{\B}-\c(\llambda(t),\z)=0
&&\quad \forall\z\in\Hub.
\endaligned
\end{equation}
In order to obtain the solution of Problem~\ref{pb:linI}, it remains to show
the existence of $p$.
\begin{proposition}
\label{pr:p}
Let $(\u,\w,\X)$ and $\llambda$ be the solutions of Problem~\ref{pb:inK} and
of~\eqref{eq:lambda}. Then there exists a unique $p\in L^2(0,T;\Ldo)$ such that
$(\u,p,\X,\w,\llambda)$ is the solution of Problem~\ref{pb:linI}.
\end{proposition}
\begin{proof}
The existence and uniqueness of $(\u,\w,\X)$ and $\llambda$ are stated in
Theorem~\ref{th:existence} and in Proposition~\ref{pr:lambda}, respectively.
The pressure $p$ can be obtained as the solution of the following equation
\begin{equation}
\label{eq:defp}
(p(t),\div\v)=\rho_f\frac d {dt}(\u(t),\v)+a(\u(t),\v)+\c(\llambda(t),\vcXb(t))
\quad\forall\v\in\Huo.
\end{equation}
In order to see that this problem defines a function $p(t)$ satisfying the
required regularity, we can use standard arguments originating from the Banach
closed range theorem (see, for instance~\cite[Theorem~4.1.4]{bbf}). We need to
show that the right-hand side of~\eqref{eq:defp} is a linear and continuous
functional on $\Huo$ belonging to the polar set of the kernel of the
divergence operator in $\Huo$. Let us denote the right-hand side
of~\eqref{eq:defp} by $\ell(\v)$; the continuity of $\ell$ can be shown as
follows
\[
\aligned
&\left\vert\int_0^T\ell(\v)\,dt\right\vert\le
\int_0^T|\ell(\v)|\,dt\\
&\qquad\le
C\left(\left\|\frac{\partial\u}{\partial t}\right\|_{L^2(H^{-1}(\Omega))}
+\|\u\|_{L^2(\Huo)}+\|\llambda\|_{L^2(\Hub)}\right)\|\v\|_{\Huo}.
\endaligned
\]
Moreover, it is clear that $\ell$ belongs to the polar set of the kernel of the
divergence operator in $\Huo$: this is exactly what is stated in the first
equation of~\eqref{eq:split}.

From the closed range theorem, it follows that there exists $p(t)$
satisfying~\eqref{eq:defp} such that
\[
\|p\|_{L^2(\Ldo)}\le(1/\beta)\|\ell\|_{L^2(H^{-1}(\Omega))},
\]
where $\beta$ is the inf-sup constant associated with the divergence operator
in $\Huo$ (see~\cite[(I.1.51) and Prop.~I.1.2]{T}).

\end{proof}  
\section*{Acknowledgments}
The authors are members of the INdAM Research group GNCS and their research
is partially supported by IMATI/CNR and by PRIN/MIUR.

\bibliographystyle{plain}
\bibliography{ref}

\begin{thebibliography}{10}

\bibitem{Adams}
R.A. Adams and J.J.F. Fournier.
\newblock {\em Sobolev spaces}, volume 140 of {\em Pure and Applied Mathematics
  (Amsterdam)}.
\newblock Elsevier/Academic Press, Amsterdam, second edition, 2003.

\bibitem{Beirao}
H.~Beir\~{a}o~da Veiga.
\newblock On the existence of strong solutions to a coupled fluid-structure
  evolution problem.
\newblock {\em J. Math. Fluid Mech.}, 6(1):21--52, 2004.

\bibitem{bbf}
D.~Boffi, F.~Brezzi, and M.~Fortin.
\newblock {\em Mixed finite element methods and applications}, volume~44 of
  {\em Springer Series in Computational Mathematics}.
\newblock Springer, Heidelberg, 2013.

\bibitem{BCG_rho}
D.~Boffi, N.~Cavallini, and L.~Gastaldi.
\newblock Finite element approach to immersed boundary method with different
  fluid and solid densities.
\newblock {\em Math. Models Methods Appl. Sci.}, 21(12):2523--2550, 2011.

\bibitem{BCG}
D.~Boffi, N.~Cavallini, and L.~Gastaldi.
\newblock The finite element immersed boundary method with distributed
  {L}agrange multiplier.
\newblock {\em SIAM J. Numer. Anal.}, 53(6):2584--2604, 2015.

\bibitem{BG_Bathe}
D.~Boffi and L.~Gastaldi.
\newblock A finite element approach for the immersed boundary method.
\newblock volume~81, pages 491--501. 2003.
\newblock In honour of Klaus-J\"{u}rgen Bathe.

\bibitem{BG}
D.~Boffi and L.~Gastaldi.
\newblock A fictitious domain approach with {L}agrange multiplier for
  fluid-structure interactions.
\newblock {\em Numer. Math.}, 135(3):711--732, 2017.

\bibitem{BGHP}
D.~Boffi, L.~Gastaldi, L.~Heltai, and C.~S. Peskin.
\newblock On the hyper-elastic formulation of the immersed boundary method.
\newblock {\em Comput. Methods Appl. Mech. Engrg.}, 197(25-28):2210--2231,
  2008.

\bibitem{wolf}
D.~Boffi, L.~Gastaldi, and S.~Wolf.
\newblock Higher-order time-stepping schemes for fluid-structure interaction
  problems, 2019.

\bibitem{Boulakia2017}
M.~Boulakia and S.~Guerrero.
\newblock On the interaction problem between a compressible fluid and a
  {S}aint-{V}enant {K}irchhoff elastic structure.
\newblock {\em Adv. Differential Equations}, 22(1-2):1--48, 2017.

\bibitem{Boulakia2019}
M.~Boulakia, S.~Guerrero, and T.~Takahashi.
\newblock Well-posedness for the coupling between a viscous incompressible
  fluid and an elastic structure.
\newblock {\em Nonlinearity}, 32(10):3548--3592, 2019.

\bibitem{CDEG}
A.~Chambolle, B.~Desjardins, M.~J. Esteban, and C.~Grandmont.
\newblock Existence of weak solutions for the unsteady interaction of a viscous
  fluid with an elastic plate.
\newblock {\em J. Math. Fluid Mech.}, 7(3):368--404, 2005.

\bibitem{CSMT}
C.~Conca, J.~San Mart\'{\i}n~H., and M.~Tucsnak.
\newblock Existence of solutions for the equations modelling the motion of a
  rigid body in a viscous fluid.
\newblock {\em Comm. Partial Differential Equations}, 25(5-6):1019--1042, 2000.

\bibitem{Coutand2005}
D.~Coutand and S.~Shkoller.
\newblock Motion of an elastic solid inside an incompressible viscous fluid.
\newblock {\em Arch. Ration. Mech. Anal.}, 176(1):25--102, 2005.

\bibitem{Coutand2006}
D.~Coutand and S.~Shkoller.
\newblock The interaction between quasilinear elastodynamics and the
  {N}avier-{S}tokes equations.
\newblock {\em Arch. Ration. Mech. Anal.}, 179(3):303--352, 2006.

\bibitem{DE1999}
B.~Desjardins and M.~J. Esteban.
\newblock Existence of weak solutions for the motion of rigid bodies in a
  viscous fluid.
\newblock {\em Arch. Ration. Mech. Anal.}, 146(1):59--71, 1999.

\bibitem{DE2000}
B.~Desjardins and M.~J. Esteban.
\newblock On weak solutions for fluid-rigid structure interaction: compressible
  and incompressible models.
\newblock {\em Comm. Partial Differential Equations}, 25(7-8):1399--1413, 2000.

\bibitem{DEGLT}
B.~Desjardins, M.~J. Esteban, C.~Grandmont, and P.~Le~Tallec.
\newblock Weak solutions for a fluid-elastic structure interaction model.
\newblock {\em Rev. Mat. Complut.}, 14(2):523--538, 2001.

\bibitem{DGHL}
Q.~Du, M.~D. Gunzburger, L.~S. Hou, and J.~Lee.
\newblock Analysis of a linear fluid-structure interaction problem.
\newblock {\em Discrete Contin. Dyn. Syst.}, 9(3):633--650, 2003.

\bibitem{F}
E.~Feireisl.
\newblock On the motion of rigid bodies in a viscous compressible fluid.
\newblock {\em Arch. Ration. Mech. Anal.}, 167(4):281--308, 2003.

\bibitem{FO}
F.~Flori and P.~Orenga.
\newblock On a nonlinear fluid-structure interaction problem defined on a
  domain depending on time.
\newblock {\em Nonlinear Anal.}, 38(5, Ser. B: Real World Appl.):549--569,
  1999.

\bibitem{GM}
C.~Grandmont and Y.~Maday.
\newblock Existence for an unsteady fluid-structure interaction problem.
\newblock {\em M2AN Math. Model. Numer. Anal.}, 34(3):609--636, 2000.

\bibitem{GLS}
M.~D. Gunzburger, H.-C. Lee, and G.~A. Seregin.
\newblock Global existence of weak solutions for viscous incompressible flows
  around a moving rigid body in three dimensions.
\newblock {\em J. Math. Fluid Mech.}, 2(3):219--266, 2000.

\bibitem{HS}
K.-H. Hoffmann and V.~N. Starovoitov.
\newblock On a motion of a solid body in a viscous fluid. {T}wo-dimensional
  case.
\newblock {\em Adv. Math. Sci. Appl.}, 9(2):633--648, 1999.

\bibitem{LM}
J.-L. Lions and E.~Magenes.
\newblock {\em Non-homogeneous boundary value problems and applications. {V}ol.
  {I}}.
\newblock Springer-Verlag, New York-Heidelberg, 1972.
\newblock Translated from the French by P. Kenneth, Die Grundlehren der
  mathematischen Wissenschaften, Band 181.

\bibitem{LW}
C.~Liu and N.~J. Walkington.
\newblock An {E}ulerian description of fluids containing visco-elastic
  particles.
\newblock {\em Arch. Ration. Mech. Anal.}, 159(3):229--252, 2001.

\bibitem{MQPe2000}
D.M. Mc~Queen and C.S. Peskin.
\newblock Heart simulation by an immersed boundary method with a formal
  second-order accuracy and reduced numerical viscosity.
\newblock In H.~Aref and J.W. Phillips, editors, {\em Mechanics for a New
  Millennium, Proceedings of the International Conference on Theoretical and
  Applied Mechanics (ICTAM) 2000}. Kluwer Academic Publisher, 2001.

\bibitem{MC2013b}
B.~Muha and S.~Cani\'{c}.
\newblock Existence of a weak solution to a nonlinear fluid-structure
  interaction problem modeling the flow of an incompressible, viscous fluid in
  a cylinder with deformable walls.
\newblock {\em Arch. Ration. Mech. Anal.}, 207(3):919--968, 2013.

\bibitem{MC2013a}
B.~Muha and S.~\vv{C}ani\'{c}.
\newblock A nonlinear, 3{D} fluid-structure interaction problem driven by the
  time-dependent dynamic pressure data: a constructive existence proof.
\newblock {\em Commun. Inf. Syst.}, 13(3):357--397, 2013.

\bibitem{MC2014}
B.~Muha and S.~\vv{C}ani\'{c}.
\newblock Existence of a solution to a fluid-multi-layered-structure
  interaction problem.
\newblock {\em J. Differential Equations}, 256(2):658--706, 2014.

\bibitem{MC2015}
B.~Muha and S.~\vv{C}ani\'{c}.
\newblock Fluid-structure interaction between an incompressible, viscous 3{D}
  fluid and an elastic shell with nonlinear {K}oiter membrane energy.
\newblock {\em Interfaces Free Bound.}, 17(4):465--495, 2015.

\bibitem{MC2016}
B.~Muha and S.~\vv{C}ani\'{c}.
\newblock Existence of a weak solution to a fluid-elastic structure interaction
  problem with the {N}avier slip boundary condition.
\newblock {\em J. Differential Equations}, 260(12):8550--8589, 2016.

\bibitem{newren}
E.~P. Newren, A.~L. Fogelson, R.~D. Guy, and R.~M. Kirby.
\newblock Unconditionally stable discretizations of the immersed boundary
  equations.
\newblock {\em J. Comput. Phys.}, 222(2):702--719, 2007.

\bibitem{PeAN}
C.~S. Peskin.
\newblock The immersed boundary method.
\newblock {\em Acta Numer.}, 11:479--517, 2002.

\bibitem{PeMcQ89}
C.~S. Peskin and D.~M. McQueen.
\newblock A three-dimensional computational method for blood flow in the heart.
  {I}. {I}mmersed elastic fibers in a viscous incompressible fluid.
\newblock {\em J. Comput. Phys.}, 81(2):372--405, 1989.

\bibitem{Peskin77}
C.S. Peskin.
\newblock Numerical analysis of blood flow in the heart.
\newblock {\em J. Computational Phys.}, 25(3):220--252, 1977.

\bibitem{QTV}
A.~Quarteroni, M.~Tuveri, and A.~Veneziani.
\newblock Computational vascular fluid dynamics: problems, models and methods.
\newblock {\em Comput Visual Sci}, 2:163--197, 2000.

\bibitem{RaymondVanni2014}
J.-P. Raymond and M.~Vanninathan.
\newblock A fluid-structure model coupling the {N}avier-{S}tokes equations and
  the {L}am\'{e} system.
\newblock {\em J. Math. Pures Appl. (9)}, 102(3):546--596, 2014.

\bibitem{Serre}
D.~Serre.
\newblock Chute libre d'un solide dans un fluide visqueux incompressible.
  {E}xistence.
\newblock {\em Japan J. Appl. Math.}, 4(1):99--110, 1987.

\bibitem{Ta}
T.~Takahashi.
\newblock Analysis of strong solutions for the equations modeling the motion of
  a rigid-fluid system in a bounded domain.
\newblock {\em Adv. Differential Equations}, 8(12):1499--1532, 2003.

\bibitem{TaTu}
T.~Takahashi and M.~Tucsnak.
\newblock Global strong solutions for the two-dimensional motion of an infinite
  cylinder in a viscous fluid.
\newblock {\em J. Math. Fluid Mech.}, 6(1):53--77, 2004.

\bibitem{T}
R.~Temam.
\newblock {\em Navier-{S}tokes equations}, volume~2 of {\em Studies in
  Mathematics and its Applications}.
\newblock North-Holland Publishing Co., Amsterdam-New York, revised edition,
  1979.
\newblock Theory and numerical analysis, With an appendix by F. Thomasset.

\end{thebibliography}

\end{document}